\documentclass[11pt]{article}
\usepackage[utf8]{inputenc}

\usepackage{geometry} 
\usepackage{graphicx} 

\usepackage{tabularx}
\newcolumntype{Z}{>{\centering\let\newline\\\arraybackslash\hspace{0pt}}X}
\usepackage{ulem}
\usepackage{booktabs} 
\usepackage{array} 
\usepackage{paralist} 
\usepackage{verbatim} 
\usepackage{subfig} 
\usepackage{hyperref}
\usepackage{amsmath}
\usepackage{amssymb}
\usepackage{authblk}
\usepackage{stmaryrd}
\usepackage{enumitem}
\setlist{nosep}
\usepackage{bbm}
\usepackage{mathtools}
\usepackage{textcomp}
\usepackage{xcolor}
\usepackage{listings}
\usepackage{float}
\usepackage{centernot}
\usepackage[mathscr]{euscript}
\usepackage{tikz-cd}
\usepackage{rotating}

\newcommand{\Set}{\mathbf{Set}}
\newcommand{\Setswith}[1]{\PSh(#1)}

\newcommand{\HOM}{\mathcal{H}\! \mathit{om}}

\DeclareMathOperator{\PSh}{\mathbf{PSh}}

\usepackage{amsthm}
\newtheorem{thm}{Theorem}

\newtheorem{proposition}[thm]{Proposition}

\newtheorem{corollary}[thm]{Corollary}

\theoremstyle{definition}

\theoremstyle{remark}

\usepackage{fancyhdr} 
\usepackage{sectsty}
\allsectionsfont{\sffamily\mdseries\upshape} 

\usepackage[nottoc,notlof,notlot]{tocbibind} 
\usepackage[titles,subfigure]{tocloft} 

\let\theta\vartheta
\let\emph\textit

\date{\vspace{-2em}} 

\title{An essential, hyperconnected, local geometric morphism that is not locally connected}
\author{Jens Hemelaer \thanks{Department of Mathematics, University of Antwerp, 
 Middelheimlaan 1, B-2020 Antwerp (Belgium) \\ email: jens.hemelaer@uantwerpen.be} \\ Morgan Rogers \thanks{Universit\`a degli Studi dell{'}Insubria, Via Valleggio n. 11, 22100 Como CO \\ Marie Sklodowska-Curie fellow of the Istituto Nazionale di Alta Matematica \\ email: mrogers@uninsubria.it}}

\begin{document}

\maketitle


\begin{abstract}
We give an example of an essential, hyperconnected, local geometric morphism that is not locally connected, arising from our work-in-progress on geometric morphisms $\Setswith{M} \to \Setswith{N}$, where $M$ and $N$ are monoids.
\end{abstract}


Thomas Streicher asked on the category theory mailing list whether every essential, hyperconnected, local geometric morphism is automatically locally connected. We show that this is not the case, by providing a counterexample. Our counterexample arises from our earlier work \cite{hr1} and work-in-progress regarding properties of geometric morphisms $\Setswith{M} \to \Setswith{N}$ for monoids $M$ and $N$.

We thank Thomas Streicher for the interesting question and the subsequent discussion regarding this counterexample.

The second named author was supported in this work by INdAM and the Marie Sklodowska-Curie Actions as a part of the \textit{INdAM Doctoral Programme in Mathematics and/or Applications Cofunded by Marie Sklodowska-Curie Actions}.

\section*{The counterexample} \label{sec:example}

Let $M$ be the monoid with presentation $\langle{e,x : e^2 = e, xe = x}\rangle$. Note that each element of $M$ can be written as either $x^n$ or $ex^n$ for some $n \in \{0,1,2,\dots\}$. Further, let $N$ be the free monoid on one variable $a$, so $N = \{1, a, a^2,\dots \}$. Consider the monoid morphism $\phi: M \to N$ which on generators is given by $\phi(e)=1$ and $\phi(x) = a$. If we interpret $M$ and $N$ as categories, then $\phi$ is a functor. There is an induced essential geometric morphism
\begin{equation*}
\begin{tikzcd}
\Setswith{M} \ar[r,"{f}"] & \Setswith{N}
\end{tikzcd}
\end{equation*}
given by functors
\begin{equation*}
\begin{tikzcd}
\Setswith{M} \ar[r,"{f_*}"', bend right] \ar[r,"{f_!}", bend left] &
\Setswith{N} \ar[l,"{f^*}"]
\end{tikzcd}
\end{equation*}
with the following description, for $X$ in $\Setswith{M}$ and $Y$ in $\Setswith{N}$:
\begin{itemize}
\item $f_!(X) \simeq X \otimes_M N$ where $N$ has left $M$-action defined by $m \cdot n = \phi(m)n$ for $m \in M$ and $n \in N$, and right $N$-action defined by multiplication;
\item $f^*(Y) \simeq Y$ with right $M$-action defined as $y \cdot m = y \cdot \phi(m)$ for $y \in Y$ and $m \in M$;
\item $f_*(Y) \simeq \HOM_M(N,Y)$, where $N$ has right $M$-action given by $n \cdot m = n\phi(m)$ for $n \in N$ and $m \in M$, and $\HOM_M(N,Y)$ is the set of morphisms of right $M$-sets $g: N \to Y$; the right $N$-action on $\HOM_M(N,Y)$ is defined as $(g \cdot n)(n') = g(nn')$ for $g \in \HOM_M(N,Y)$ and $n,n' \in N$. 
\end{itemize}

For definitions and background regarding tensor products and Hom-sets, in the context of sets with a monoid action, we refer to \cite[Subsection 1.2]{hr1}.

\begin{proposition} \label{prop:hyperconnected}
The geometric morphism $f$ is hyperconnected.
\end{proposition}
\begin{proof}
This follows from $\phi$ being surjective, see \cite[Example A.4.6.9]{Ele}.
\end{proof}

\begin{proposition}
The geometric morphism $f$ is local.
\end{proposition}
\begin{proof}
First, we claim that $f_*$ has a right adjoint. By the Special Adjoint Functor Theorem, it is enough to show that $f_*$ preserves colimits. This is equivalent to the functor $\HOM_M(N,-) : \Setswith{M} \to \Set$ preserving colimits, since colimits in $\Setswith{N}$ are computed on underlying sets. By \cite[Lemma 4.1]{TDMA} this is equivalent to $N$ being an indecomposable projective right $M$-set. To see that $N$ is an indecomposable projective right $M$-set, consider the map $\phi : M \to N$. It is a morphism of right $M$-sets, and it has a section $\sigma : N \to M$ defined by $\sigma(a^n) = ex^n$. The function $\sigma$ is a morphism of right $M$-sets, so $N$ is a retract of $M$. Since $M$ is indecomposable projective, and $N$ is a retract of $M$, we have that $N$ is indecomposable projective as well.

In order for $f$ to be local, we additionally need that $f$ is connected. But this follows from $f$ being hyperconnected, see Proposition \ref{prop:hyperconnected}.
\end{proof}

Suppose that $f$ is locally connected. Then it follows from the above and \cite[Proposition 3.5]{johnstone-plc} that $f_!$ preserves binary products. We will show that this is not the case, and as a result $f$ is not locally connected.

\begin{proposition} \label{prop:does-not-preserve-binary-products}
The functor $f_!$ does not preserve binary products.
\end{proposition}
\begin{proof}
Consider the right $M$-sets $M$ and $N$, with the right $M$-action on $M$ given by multiplication, and the right $M$-action on $N$ given by $n \cdot m = n\phi(m)$. We claim that the natural comparison map
\begin{equation*}
f_!(M \times N) \longrightarrow f_!(M) \times f_!(N)
\end{equation*}
is not injective. For a general right $M$-set $X$, we have $f_!(X) \simeq X \otimes_M N$, where the left $M$-action on $N$ is given by $m \cdot n = \phi(m)n$. The right $N$-action on $X \otimes_M N$ is given by $(x \otimes n) \cdot n' = x \otimes nn'$. In particular, $f_!(M) \simeq M \otimes_M N \simeq N$. Further, the map $N \otimes_M N \to N$, $n \otimes n' \mapsto nn'$ is an isomorphism, so $f_!(N) \simeq N$ as well.

Using the above, we can rewrite the comparison map as
\begin{equation*}
\begin{split}
\alpha : (M \times N) \otimes_M N \longrightarrow N \times N \\
(m,n) \otimes n' \mapsto (\phi(m)n',nn')
\end{split}
\end{equation*}
Note that $\alpha((x,1)\otimes 1) = (a,1) = \alpha((ex,1)\otimes 1)$. We will show that $(x,1)\otimes 1 \neq (ex,1)\otimes 1$, which implies that $\alpha$ is not injective. First we show that $(x,1)$ and $(ex,1)$ are in different components of $M \times N$. Write:
\begin{align*}
& A' = \{ (x,1), (x^2,a), (x^3,a^2), (x^4,a^3), \dots  \} \\
& A'' = \{ (ex,1), (ex^2,a), (ex^3,a^2), (ex^4,a^3), \dots \} \\
& B = \{ (m,n) \in M \times N : \phi(m) \neq an \}
\end{align*}
It can be checked directly that $A'$ and $A''$ are sub-$M$-sets of $M \times N$. From cancellativity of $N$ it follows that $B \subseteq M \times N$ is a sub-$M$-set as well. So we can write $M \times N$ as a direct sum (coproduct) of right $M$-sets 
\begin{equation*}
M \times N = A' \sqcup A'' \sqcup B. 
\end{equation*}
We have $(x,1) \in A'$ and $(ex,1) \in A''$. Because $f_!$ preserves colimits, we find that
\begin{equation*}
(M \times N) \otimes_M N \simeq f_!(M \times N) \simeq f_!(A') \sqcup f_!(A'') \sqcup f_!(B).
\end{equation*}
Since $(x,1)\otimes 1 \in f_!(A')$ and $(ex,1) \otimes 1 \in f_!(A'')$, we have that $(x,1)\otimes 1 \neq (ex,1) \otimes 1$. As a result, $\alpha$ is not injective.
\end{proof}

From the discussion before Proposition \ref{prop:does-not-preserve-binary-products} it now follows that:

\begin{corollary}
The geometric morphism $f$ is not locally connected.
\end{corollary}

\bibliographystyle{amsplainarxiv}
\bibliography{../monoidprops}

\providecommand{\bysame}{\leavevmode\hbox to3em{\hrulefill}\thinspace}
\providecommand{\MR}{\relax\ifhmode\unskip\space\fi MR }
\providecommand{\MRhref}[2]{%
  \href{http://www.ams.org/mathscinet-getitem?mr=#1}{#2}
}
\providecommand{\href}[2]{#2}
\begin{thebibliography}{1}

\bibitem{hr1}
J.~Hemelaer and M.~Rogers, \emph{{Monoid Properties as Invariants of Toposes of
  Monoid Actions}}, preprint (2020), \href {http://arxiv.org/abs/2004.10513}
  {\path{arXiv:2004.10513}}.

\bibitem{Ele}
P.T. Johnstone, \emph{{{Sketches of an Elephant: A Topos Theory Compendium}}},
  {{The Clarendon Press, Oxford University Press, Oxford}}, 2002.

\bibitem{johnstone-plc}
\bysame, \emph{Remarks on punctual local connectedness}, Theory Appl. Categ.
  \textbf{25} (2011), No. 3, 51--63. \MR{2805745}

\bibitem{lawvere}
F.~William Lawvere, \emph{Axiomatic cohesion}, Theory Appl. Categ. \textbf{19}
  (2007), No. 3, 41--49. \MR{2369017}

\bibitem{TDMA}
M.~Rogers, \emph{{Toposes} of {Discrete} {Monoid} {Actions}}, preprint (2019),
  \href {http://arxiv.org/abs/1905.10277} {\path{arXiv:1905.10277}}.

\end{thebibliography}

\end{document}